\documentclass{amsart}
\usepackage{amsfonts}
\usepackage{amsmath,amssymb}
\usepackage{amsthm}
\usepackage{amscd}
\usepackage{graphics}
\usepackage{graphicx}

\theoremstyle{remark}{

\newtheorem{Rem}{{\rm Remark}}

}
\theoremstyle{plain}
{

\newtheorem{Thm}{Theorem}
\newtheorem*{MainThm}{Main Theorem}

}

\begin{document}
\title[A study of Reeb spaces of non-proper smooth functions]{On Reeb spaces of non-proper smooth functions which are $1$-dimensional CW complexes}
\author{Naoki kitazawa}

\address{Osaka Central Advanced Mathematical Institute (OCAMI) \\
3-3-138 Sugimoto, Sumiyoshi-ku Osaka 558-8585
TEL: +81-6-6605-3103
}
\email{naokikitazawa.formath@gmail.com}
\urladdr{https://naokikitazawa.github.io/NaokiKitazawa.html}
\maketitle
\begin{abstract}
We discuss {\it Reeb spaces} of non-proper smooth functions. 
{\it {\rm (}Non-{\rm )}proper} maps are maps between topological spaces the preimages of compact sets by which are compact (resp. may not be compact). Their {\it Reeb spaces} are the spaces of all connected components of their level sets and with the natural quotient topologies.

In proper cases, explicit theory is developing since the 1950s, and related to Morse(-Bott) function theory. In such cases we have graphs naturally. Recently, general topological studies on Reeb spaces in proper cases with combinatorial ones are actively developing, mainly due to Gelbukh and Saeki. They are at most $1$-dimensional in certain tame cases. We extend a theorem by Saeki in the 2020s: the Reeb space of a smooth function on a closed manifold is naturally a graph if and only if its critical value set is finite.

\end{abstract}
\section{Introduction.}
\label{sec:1}
\subsection{Reeb spaces and Reeb graphs of smooth functions on manifolds with no boundary.}
\label{subsec:1.1}
{\it Reeb spaces} of smooth real-valued functions have already appeared in the 1950s (\cite{reeb}). They are spaces of all connected components of their level sets and with the natural quotient topologies. For a continuous real-valued function $c:X \rightarrow \mathbb{R}$ on a topological space $X$, we can define the equivalence relation ${\sim}_c$ on $X$ by the rule $x_1 {\sim}_c x_2$ if and only if they belong to a same connected component of the preimage $c^{-1}(y)$ (a {\it level set} of $c$). We have the quotient map $q_c:X \rightarrow R_c$ and the unique continuous function $\bar{c}:R_c \rightarrow \mathbb{R}$ with $c=\bar{c} \circ q_c$. This notion is closely related to Morse(-Bott) theory. Such spaces are fundamental tools in understanding various geometry of the manifolds $X$. They are fundamental and important in related theory and applications to differential topology and various geometry.

We do not need to understand related mathematics such as Morse(-Bott) functions and
 singularity theory of differentiable functions and maps well, here. However, we present related textbooks. For Morse functions, see \cite{milnor1, milnor2} and for Morse-Bott functions, see \cite{banyagahurtubise}. For singularity theory of differentiable maps, see \cite{golubitskyguillemin}.

For a differentiable function $c:X \rightarrow \mathbb{R}$, its {{\it critical} point $p \in X$ is of fundamental notions and it is a point where the differential of $c$ vanishes. 
The set of all critical point of $c$, denoted by $S(c)$, is the {\it critical set} of $c$ and $c(S(c))$ is the {\it critical value set} of $c$.

For a smooth function $c:X \rightarrow \mathbb{R}$ on a closed manifold $X$, if $R_c$ is
 at most $1$-dimensional and homeomorphic to a graph with $q_c(S(c))$ being the vertex set of the graph ${\mathcal{R}}_c=R_c$, then it is the {\it Reeb graph} of $c$, They are indeed graphs in the case
 of Morse(-Bott) functions $c$ and it is a kind of classical theory presented in \cite{reeb} with \cite{izar}, and for the Morse-Bott case see \cite{martinezalfaromezasarmientooliveira}, for example.
 
Recently, topological theory and combinatorial theory of Reeb spaces are developing, mainly due to Gelbukh and Saeki. Surprisingly such theory is new. There, for example, fundamental arguments on topological spaces and advanced theory is also important. These studies have been mainly for {\it proper} functions $c:X \rightarrow \mathbb{R}$, defined
 as functions such that the preimage $c^{-1}(K)$ is compact for any compact set $K \subset \mathbb{R}$. We are focused on {\it non-proper} functions, defined as real-valued functions on topological spaces which are not proper. For non-proper cases, we encounter so-called wild situations frequently and difficulty lies here. Recently, Gelbukh is actively studying non-proper cases and sufficient conditions for Reeb spaces to be Hausdorff, $1$-dimensional, and so on, have been studied, in \cite{gelbukh1, gelbukh2, gelbukh3, gelbukh4}.

Here, we are focused on a theorem on the graph structures of Reeb spaces of proper smooth functions, Theorem \ref{thm:1} (\cite[Theorem 3.1]{saeki1}: see also \cite[Theorems 2.1 and 2.8]{saeki2}).

Here, graphs and more generally, CW complexes (the closures of each $1$-cells of which may not be compact or which may not be locally finite) are important. For elementary theory on CW complexes, see \cite{hatcher} for example. It is also important that these spaces are Hausdorff. If the Reeb space $R_c$ of a smooth real-valued function $c:X \rightarrow \mathbb{R}$ on $X$ with no boundary is (Hausdorff and) regarded to be a at most $1$-dimensional cell complex, then we also use ${\mathcal{R}}_c$ for the Reeb space with the structure of the cell complex.

\begin{Thm}[\cite{saeki1}]
\label{thm:1}
For a smooth function $c:X \rightarrow \mathbb{R}$ on a closed manifold, 
${\mathcal{R}}_c=R_c$ is a finite graph whose vertex set is $q_c(S(c))$ {\rm (}, or the Reeb graph ${\mathcal{R}}_c=R_c$ of $c$,{\rm )} if and only if the critical value set $c(s(c))$ of $c$ is finite.

More generally, for a proper smooth function $c:X \rightarrow \mathbb{R}$ on a manifold $X$ with no boundary, ${\mathcal{R}}_c=R_c$ is a at most $1$-dimensional CW complex the set of all of whose 0-cells is $q_c(S(c))$ and which is locally finite, if and only if $c(S(c))$ is discrete and closed in $\mathbb{R}$.
\end{Thm}
\subsection{Reeb spaces of non-proper continuous or smooth functions, our main result, and the story of the present paper.}
Reeb spaces of non-proper smooth functions are also explicitly studied in \cite{kitazawa}. This is related to so-called reconstruction of nice smooth functions on manifolds with no boundary, founded in \cite{sharko} and followed in \cite{masumotosaeki, michalak} for example. 

Our main result is a non-proper version of Theorem \ref{thm:1}. In the next section, we present it as Main Theorem with its proof. Our main strategy consists of application of important facts and arguments on topological spaces such as Theorem \ref{thm:2}. Theorem \ref{thm:1} is also respected in the proof, where we do not need to understand arguments in the original article \cite{saeki1,saeki2} well. We also explain related explicit cases and studies, as a kind of short appendices.

\section{Main Theorem.}
\subsection{Preliminaries.}

First, we introduce several fundamental notation for Main Theorem.

For the restriction of a map $c:X \rightarrow Y$ to a subset $Z\subset X$, we use $c {\mid}_Z$. For the closure of $Y \subset X$ in a topological space $X$, we use ${\overline{Y}}^X$.

The following are of fundamental propositions and theorems and they are presented in some Gelbukh's papers. They are also used in the present paper. See also \cite{willard} for other fundamental propositions and theorems on topological spaces including classical ones.
\begin{Thm}
\label{thm:2}
\begin{enumerate}
\item \label{thm:2.1} {\rm (}\cite[Theorem 3.2]{gelbukh1}, \cite[Theorem 27.12]{willard}, and so on.{\rm )} Every quotient space of a locally connected space is also locally connected.

\item \label{thm:2.2} {\rm (}\cite[Theorem 6.2]{gelbukh4}{\rm )} For a continuous function $c:X \rightarrow \mathbb{R}$ whose level sets are always connected, $R_c$ is Hausdorff. Furthermore, here, if $\bar{c}:R_c \rightarrow \mathbb{R}$ is an open map, then $R_c$ is homeomorphic to $c(X) \subset \mathbb{R}$.
\item \label{thm:2.3} {\rm (}\cite[Theorem 9.1]{gelbukh4}{\rm )} For a continuous function $c:X \rightarrow \mathbb{R}$ on a topological manifold $X$, $R_c$ is Hausdorff if and only if at least one of the following three holds.
\begin{itemize}
\item $X$ is compact.
\item Level sets of $c$ are always compact.
\item Level sets of $c$ are always connected. 
\end{itemize} 
\item \label{thm:2.4} {\rm (}\cite[Theorem 7.5]{gelbukh1}{\rm )} For a smooth function $c:X \rightarrow \mathbb{R}$ on a closed and connected manifold which is not constant, $R_c$ is a $1$-dimensional, compact, connected, and locally connected space. In short, it is a {\rm Peano continuum}. Furthermore, $R_c$ is also metrizable in this case.

\end{enumerate}
\end{Thm}
Remark \ref{rem:1} is on generalizations of the notion of Reeb space and explicit facts on the topology such as Theorem \ref{thm:2} for general continuous maps $c:X \rightarrow Y$.
\begin{Rem}
	\label{rem:1}
We can define the {\it Reeb space} of a (continuous) map $c:X \rightarrow Y$ between topological spaces for a general space $Y$, in the same way. For the general situation, similar facts hold, in considerable cases. 
\end{Rem}

\subsection{Main Theorem and its proof.}
Main Theorem is as follows.
\begin{MainThm}

For a non-proper and non-constant smooth function $c:X \rightarrow \mathbb{R}$ on a manifold $X$ with no boundary, consider the following family of non-empty open subsets $\{U_{{\lambda}_{\rm P}}\}_{{\lambda}_{\rm P} \in {\Lambda}_{\rm P}}$ and non-empty open and connected subsets $\{U_{{\lambda}_{\rm NP}}\}_{{\lambda}_{\rm NP} \in {\Lambda}_{\rm NP}}$ of $R_c$.
\begin{itemize}
	\item These sets $U_{{\lambda}_{\rm P}}$ and $U_{{\lambda}_{\rm NP}}$ are mutually disjoint.
\item The restriction $c {\mid}_{{q_c}^{-1}(U_{{\lambda}_{\rm P}})}$ is proper.
\item The restriction $c {\mid}_{{q_c}^{-1}(U_{{\lambda}_{\rm NP}})}$ has no critical point, and its level sets are always non-compact and connected.
\end{itemize}
We define a closed subset $V_{R_c,{\Lambda}_{\rm P}, {\Lambda}_{\rm NP}}:=R_c-({\sqcup}_{{\lambda}_{\rm P} \in {\Lambda}_{\rm P}} (U_{{\lambda}_{\rm P}}) \sqcup {\sqcup}_{{\lambda}_{\rm NP} \in {\Lambda}_{\rm NP}} (U_{{\lambda}_{\rm NP}}))\subset R_c$.

Then we have the following two.
\begin{enumerate}
\item \label{thm:3.1} $R_c$ is Hausdorff if and only if for each point $v \in \bar{c}(V_{R_c,{\Lambda}_{\rm P}, {\Lambda}_{\rm NP}})$, distinct two points in the level set $c^{-1}(v)$ of $c$ can be always separated by suitable open neighborhoods in $R_c$.
\item \label{thm:3.2} Suppose that $R_c$ is Hausdorff, and that ${\overline{U_{{\lambda}_{\rm P}}}}^{R_c}$ is always compact. 
We have a $1$-dimensional CW complex ${\mathcal{R}}_{c,V_{R_c,{\Lambda}_{\rm P}, {\Lambda}_{\rm NP}}}$ whose underlying set is $R_c$, the set of all 0-cells of which is the closed and discrete subset ${\bar{c}}^{-1}(\bar{c}(V_{R_c,{\lambda}_{\rm P}, {\lambda}_{\rm NP}})) \bigcup q_c(S(c))$, and which may not be locally finite, if and only if both the following two hold.  

\begin{enumerate}
\item \label{thm:3.2.1} ${\bar{c}}^{-1}(\bar{c}(V_{R_c,{\Lambda}_{\rm P},{\Lambda}_{\rm NP}}))$ is closed and discrete in $R_c$. 
\item \label{thm:3.2.2} $c(S(c {\mid}_{{q_c}^{-1}(U_{{\lambda}_{\rm P}})}))$ are always finite sets.
\end{enumerate}
\end{enumerate}
\end{MainThm}

\begin{proof}
STEP 1 Main Theorem (\ref{thm:3.1}). \\

We prove Main Theorem (\ref{thm:3.1}).

First, suppose that for some point $v \in \bar{c}(V_{R_c,{\Lambda}_{\rm P}, {\Lambda}_{\rm NP}})$, some distinct two points in the level set $c^{-1}(v)$ of $c$ are
not separated by any suitable open neighborhoods in $R_c$ (remember the function $\bar{c}:R_c \rightarrow \mathbb{R}$ and the relation $c=\bar{c} \circ q_c$ again). In this case, $R_c$ is not Hausdorff.

Hereafter, suppose that for each point $v \in \bar{c}(V_{R_c,{\Lambda}_{\rm P}, {\Lambda}_{\rm NP}}) \subset \mathbb{R}$,
any distinct two points in the level set $c^{-1}(v)$ of $c$ are separated by suitable open neighborhoods in $R_c$. We have that (${\sqcup}_{{\lambda}_{\rm P} \in {\Lambda}_{\rm P}} (U_{{\lambda}_{\rm P}}) \sqcup {\sqcup}_{{\lambda}_{\rm NP} \in {\Lambda}_{\rm NP}} (U_{{\lambda}_{\rm NP}})) \subset R_c$ is Hausdorff, by our definitions and Theorem \ref{thm:2} (\ref{thm:2.3}). By the argument and this assumption on $v \in \bar{c}(V_{R_c,{\Lambda}_{\rm P}, {\Lambda}_{\rm NP}})$ and the level set $c^{-1}(v)$ of $c$, arbitrary two distinct points in $R_c$ are separated by some open sets in $R_c$ and $R_c$ is Hausdorff.

This completes STEP 1. \\
\ \\
STEP 2 Main Theorem (\ref{thm:3.2}).\\

We prove Main Theorem (\ref{thm:3.2}). $R_c$ is assumed to be Hausdorff. \\ 
\ \\
\noindent STEP 2-1 The condition (\ref{thm:3.2.1}) should be assumed.\\

First, suppose that the condition (\ref{thm:3.2.1}) does not hold. In this case, we cannot have any $1$-dimensional CW complex ${\mathcal{R}}_{c,V_{R_c,{\Lambda}_{\rm P}, {\Lambda}_{\rm NP}}}$, by
the definition of a CW complex.

We prove the theorem under the assumption that the condition (\ref{thm:3.2.1}) holds. \\

\noindent STEP 2-2 The condition (\ref{thm:3.2.2}) should be assumed.\\

We discuss the condition (\ref{thm:3.2.2}).

The restriction $c {\mid}_{{q_c}^{-1}(U_{{\lambda}_{\rm NP}})}$ has only non-compact and connected level sets and no critical point. 
We investigate the function $\bar{c {\mid}_{{q_c}^{-1}(U_{{\lambda}_{\rm NP}})}}:R_{c {\mid}_{{q_c}^{-1}(U_{{\lambda}_{\rm NP}})}} \subset R_C \rightarrow \mathbb{R}$. For a point $p$ in each level set of $c {\mid}_{{q_c}^{-1}(U_{{\lambda}_{\rm NP}})}$, we can find a small open set $U_p$ which is diffeomorphic to an open disk in the manifold ${q_c}^{-1}(U_{{\lambda}_{\rm NP}})$ and whose closure ${\overline{U_p}}^{{q_c}^{-1}(U_{{\lambda}_{\rm NP}})}$ is diffeomorphic to a closed disk embedded there and compact. We also have a vector field on ${\overline{U_p}}^{{q_c}^{-1}(U_{{\lambda}_{\rm NP}})}$ where the differential of $c {\mid}_{{q_c}^{-1}(U_{{\lambda}_{\rm NP}})}$ does not vanish.
This means that $\bar{c {\mid}_{{q_c}^{-1}(U_{{\lambda}_{\rm NP}})}}:R_{c {\mid}_{{q_c}^{-1}(U_{{\lambda}_{\rm NP}})}} \rightarrow \mathbb{R}$ is an open map. Furthermore, by Theorem \ref{thm:2} (\ref{thm:2.2}),
$U_{{\lambda}_{\rm NP}}$ is homeomorphic to the disjoint union of open intervals in $\mathbb{R}$, and an open interval by the assumption of its connectedness. In addition, its closure ${\overline{U_{{\lambda}_{\rm NP}}}}^{R_c}$ is obtained by adding no set or a suitable non-empty set to $U_{{\lambda}_{\rm NP}}$.

The restriction $c {\mid}_{{q_c}^{-1}(U_{{\lambda}_{\rm P}})}$ is proper. 
By so-called Ehresmann fibration theorem or the arguments as ones in investigating the function $c {\mid}_{{q_c}^{-1}(U_{{\lambda}_{\rm NP}})}$ and the function $\bar{c {\mid}_{{q_c}^{-1}(U_{{\lambda}_{\rm NP}})}}:R_{c {\mid}_{{q_c}^{-1}(U_{{\lambda}_{\rm NP}})}} \subset R_C \rightarrow \mathbb{R}$, the complementary set of $U_{{\lambda}_{\rm P}} \bigcap q_c(S(c))$ in $U_{{\lambda}_{\rm P}}$ is homeomorphic to the disjoint union of finitely many open intervals in $\mathbb{R}$.

Suppose that some $c(S(c {\mid}_{{q_c}^{-1}(U_{{\lambda}_{\rm P}})}))$ are infinite sets: the condition (\ref{thm:3.2.2}) does not hold. Remember that all ${\overline{U_{{\lambda}_{\rm P}}}}^{R_c}$ are assumed to be compact, and that they are Hausdorff. We can see that ${\bar{c}}^{-1}(\bar{c}(V_{R_c,{\lambda}_{\rm P}, {\lambda}_{\rm NP}})) \bigcup q_c(S(c))$ is not discrete in $R_c$. Thus, we cannot have any $1$-dimensional CW complex ${\mathcal{R}}_{c,V_{R_c,{\Lambda}_{\rm P}, {\Lambda}_{\rm NP}}}$ the set of all $0$-cells of which is ${\bar{c}}^{-1}(\bar{c}(V_{R_c,{\lambda}_{\rm P}, {\lambda}_{\rm NP}})) \bigcup q_c(S(c))$.

Based on this, hereafter, suppose that $c(S(c {\mid}_{{q_c}^{-1}(U_{{\lambda}_{\rm P}})}))$ are always finite. This means that the condition (\ref{thm:3.2.2}) holds. It also holds that ${\mathcal{R}}_{c {\mid}_{{q_c}^{-1}(U_{{\lambda}_{\rm P}})}}$ is regarded to be a $1$-dimensional CW complex which is finite and locally finite, by Theorem \ref{thm:1}. \\

\noindent STEP 2-3 On ${\bar{c}}^{-1}(\bar{c}(V_{R_c,{\lambda}_{\rm P}, {\lambda}_{\rm NP}})) \bigcup q_c(S(c))$ and a small (open and connected) neighborhood $N(v_{R_c})$ of $v_{R_c} \in V_{R_c,{\Lambda}_{\rm P}, {\Lambda}_{\rm NP}}$ in $R_c$.\\

First, remember the following. By STEP 2-2,  each of sets $U_{{\lambda}_{\rm P}}$ is a $1$-dimensional finite and locally finite CW complex and each of sets $U_{{\lambda}_{\rm NP}}$ is homeomorphic to open intervals.

Choose all $U_{{\lambda}_{\rm P}}$ and $U_{{\lambda}_{\rm NP}}$ whose closures $\overline{U_{{\lambda}_{\rm P}}}^{R_c}$ and $\overline{U_{{\lambda}_{\rm NP}}}^{R_c}$  contain $v_{R_c}$. 

There exists a connected open neighborhood $N(v_{R_c})$ of $v_{R_c}$ in $R_c$ containing no other point in ${\bar{c}}^{-1}(\bar{c}(v_{R_c}))$ by STEP 2-2, and any such set $N(v_{R_c})$ contains infinitely many points of $U_{{\lambda}_{\rm P}}$ ($U_{{\lambda}_{\rm NP}}$) for each of $U_{{\lambda}_{\rm P}}$ (resp. $U_{{\lambda}_{\rm NP}}$) in $R_c$.

Remember that $R_c$ is assumed to be Hausdorff. Suppose that the infinitely many points are in some compact subset $K_{\rm P}$ (resp. $K_{\rm NP}$) of each $U_{{\lambda}_{\rm P}}$ (resp. $U_{{\lambda}_{\rm NP}}$), being regarded to be a $1$-dimensional finite CW-complex and metrizable, and under this assumption, $v_{R_c}$ and the compact set $K_{\rm P}$ (resp. $K_{\rm NP}$) must be separated by their open neighborhoods in $R_c$, and this is a contradiction. 

Remember that each of sets $U_{{\lambda}_{\rm P}}$ is a $1$-dimensional finite and locally finite CW complex and that each of sets $U_{{\lambda}_{\rm NP}}$ is homeomorphic an open interval.

From the argument, there exist a connected closed neighborhood $C(U_{{\lambda}_{\rm P}} \bigcap q_f(S(f)))$ of $U_{{\lambda}_{\rm P}} \bigcap q_f(S(f))$ being also compact in $U_{{\lambda}_{\rm P}}$ and a non-empty connected subset $C(U_{{\lambda}_{\rm NP}})$ of $U_{{\lambda}_{\rm NP}}$ being also compact, and we can choose at least one open and connected set $I_{{\lambda}_{\rm P}} \subset U_{{\lambda}_{\rm P}}-C(U_{{\lambda}_{\rm P}} \bigcap q_f(S(f)))$ in each set $U_{{\lambda}_{\rm P}}$ and at least one open and connected set $I_{{\lambda}_{\rm NP}}$ in each set $U_{{\lambda}_{\rm NP}}$ with the following.
\begin{itemize}
\item The closure of $I_{{\lambda}_{\rm P}}$ in $R_c$ contains the point $v_{R_c}$. $I_{{\lambda}_{\rm P}}$ is homeomorphic to and identified with an open interval of the form $\{t \mid a_{I_{{\lambda}_{\rm P}},1}<t<a_{I_{{\lambda}_{\rm P}},2}\}$ and either the following hold.
\begin{itemize}
\item If we choose a decreasing sequence of numbers there, converging to $a_{I_{{\lambda}_{\rm P}},1}$, then it always converges to a fixed point $p_{a_{I_{{\lambda}_{\rm P}}}}$ of $U_{{\lambda}_{\rm P}}-{\bar{c}}^{-1}(\bar{c}(v_{R_c}))$, in $R_c$. If we choose an increasing sequence of numbers there, converging to $a_{I_{{\lambda}_{\rm P}},2}$, then it cannot converge to any point of $U_{{\lambda}_{\rm P}}$, in $R_c$, and furthermore, the values of $\bar{c}$ at the points of the sequence converges to $\bar{c}(v_{R_c})$.
\item If we choose an increasing sequence of numbers there, converging to $a_{I_{{\lambda}_{\rm P}},2}$, then it always converges to a fixed point in $p_{a_{I_{{\lambda}_{\rm P}}}} \in U_{{\lambda}_{\rm P}}-{\bar{c}}^{-1}(\bar{c}(v_{R_c}))$, in $R_c$. If we choose a decreasing sequence of numbers there, converging to $a_{I_{{\lambda}_{\rm P}},1}$, then it cannot converge to any point of $U_{{\lambda}_{\rm P}}$, in $R_c$, and furthermore, the corresponding sequence of the values of $\bar{c}$ at the points of the sequence converges to $\bar{c}(v_{R_c})$.
\end{itemize}

\item The closure of $I_{{\lambda}_{\rm NP}}$ in $R_c$ contains the point $v_{R_c}$. $I_{{\lambda}_{\rm NP}}$ is homeomorphic to and identified with an open interval of the form $\{t \mid a_{I_{{\lambda}_{\rm NP}},1}<t<a_{I_{{\lambda}_{\rm NP}},2}\}$ and as in the previous case, either of the corresponding two facts holds.

\end{itemize}

Consider the union $I_{v_{R_c}}$ of all such connected sets $I_{{\lambda}_{\rm P}}$  and $I_{{\lambda}_{\rm NP}}$ corresponding to $v_{R_c}$.
Suppose that one of sequences of points in $I_{v_{R_c}}$ converges to some point ${v_{R_c}}^{\prime} \in {\bar{c}}^{-1}(\bar{c}(v_{R_c}))-\{v_{R_c}\}$. 
We already have such a sequence converging to $v_{R_c}$, by our arguments. By Theorem \ref{thm:2} (\ref{thm:2.1}), $X$ is locally connected and $R_c$ is locally connected. $R_c$ is also assumed to be Hausdorff. We have a contradiction by an argument of \cite[Proof of Theorem 2.8]{saeki2}. More precisely, $v_{R_c}$ and ${v_{R_c}}^{\prime}$ are separated by some small open connected neighborhoods ${N(v_{R_c})}^{\prime}$ and ${N({v_{R_c}}^{\prime})}^{\prime}$ in $R_c$, both the intersections ${N(v_{R_c})}^{\prime} \bigcap I_{v_{R_c}}$ and ${N({v_{R_c}}^{\prime})}^{\prime} \bigcap I_{v_{R_c}}$ contain infinitely many connected components homeomorphic to an open interval, and this is a contradiction. This means that we cannot have a case ${v_{R_c}}^{\prime} \neq v_{R_c}$.

Furthermore, any sequence of elements in each of the sets $I_{{\lambda}_{\rm P}} \subset R_c$ and $I_{{\lambda}_{\rm NP}} \subset R_c$ whose corresponding sequence of the values of $\bar{c}$ converges to $\bar{c}(v_{R_c})$ must not diverge and must converge to $v_{R_c}$, essentially due to the fact that $R_c$ is Hausdorff and locally connected and the following. A neighborhood system of $v_{R_c}$ in each of $\{v_{R_c}\} \bigcup I_{{\lambda}_{\rm P}}$ and $\{v_{R_c}\} \bigcup I_{{\lambda}_{\rm NP}}$ can be chosen to be of one of the following forms, where $t_{\mathbb{Q}}$ stands for all positive rational numbers satisfying $t_{\mathbb{Q}}<a_{I_{{\lambda}_{\rm NP}},2}-a_{I_{{\lambda}_{\rm NP}},1}$.
\begin{itemize}
\item  $\{\{v_{R_c}\} \sqcup \{t \mid a_{I_{{\lambda}_{\rm P}},1}<t<a_{I_{{\lambda}_{\rm P}},1}+t_{\mathbb{Q}}\} \mid t_{\mathbb{Q}}\}$.
\item $\{\{v_{R_c}\} \sqcup \{t \mid a_{I_{{\lambda}_{\rm NP}},1}<t<a_{I_{{\lambda}_{\rm NP}},1}+t_{\mathbb{Q}}\} \mid t_{\mathbb{Q}}\}$.
\item  $\{\{v_{R_c}\} \sqcup \{t \mid a_{I_{{\lambda}_{\rm P}},2}-t_{\mathbb{Q}}<t<a_{I_{{\lambda}_{\rm P}},2}\} \mid t_{\mathbb{Q}}\}$.
\item  $\{\{v_{R_c}\} \sqcup \{t \mid a_{I_{{\lambda}_{\rm P}},2}-t_{\mathbb{Q}}<t<a_{I_{{\lambda}_{\rm P}},2}\} \mid t_{\mathbb{Q}}\}$.
\end{itemize}

We can choose a connected open neighborhood $N(v_{R_c}):=U(v_{R_c})$ of $v_{R_c}$ in $R_c$ which does not contain any other element of ${\bar{c}}^{-1}(\bar{c}(V_{R_c,{\Lambda}_{\rm P},{\Lambda}_{\rm NP}}))$, due to the assumption that ${\bar{c}}^{-1}(\bar{c}(V_{R_c,{\Lambda}_{\rm P},{\Lambda}_{\rm NP}}))$ is discrete. By the fact that $R_c$ is locally connected, we can choose such a connected open neighborhood of $v_{R_c}$ in $R_c$ disjoint from sets $U_{{\lambda}_{\rm P}}$ and $U_{{\lambda}_{\rm NP}}$ in $R_c$ whose closures ${\overline{U_{{\lambda}_{\rm P}}}}^{R_c}$ and ${\overline{U_{{\lambda}_{\rm NP}}}}^{R_c}$ do not contain $v_{R_c}$: if we cannot do this, then the set $U(v_{R_c})$ must always contain at least one connected and open set of $U(v_{R_c})$ being also regarded to be a $1$-dimensional locally finite CW complex in $U_{{\lambda}_{\rm P}}$ or $U_{{\lambda}_{\rm NP}}$ whose closures ${\overline{U_{{\lambda}_{\rm P}}}}^{R_c}$ and ${\overline{U_{{\lambda}_{\rm NP}}}}^{R_c}$ do not contain $v_{R_c}$, and we have a contradiction. For this, remember also several arguments on locally-connected spaces before. 

From the arguments, we can also choose $N(v_{R_c})=U(v_{R_c}):=I_{v_{R_C}}$.
Remember again that $q_c(S(c))-(q_c(S(c)) \bigcap V_{R_c,{\Lambda}_{\rm P}, {\Lambda}_{\rm NP}})$ is discrete and closed in $R_c-V_{R_c,{\Lambda}_{\rm P}, {\Lambda}_{\rm NP}}$ and by our arguments this is also closed and discrete in $R_c$. Remember also that $V_{R_c,{\Lambda}_{\rm P}, {\Lambda}_{\rm NP}}$ is a closed and discrete subset of $R_c$. Thus, $V_{R_c,{\lambda}_{\rm P}, {\lambda}_{\rm NP}} \bigcup q_c(S(c))$ is a closed and discrete subset of $R_c$.

Thus, we have a CW complex ${\mathcal{R}}_{c,V_{R_c,{\Lambda}_{\rm P}, {\Lambda}_{\rm NP}}}$ whose underlying set is $R_c$, the set of all $0$-cells of which is the discrete and closed subset ${\bar{c}}^{-1}(\bar{c}(V_{R_c,{\lambda}_{\rm P}, {\lambda}_{\rm NP}})) \bigcup q_c(S(c))$, and which may not be locally finite. This CW complex is $1$-dimensional since the function is non-constant. 

This completes STEP 2. \\
\ \\
This completes the proof.
\end{proof}
We close the paper by explaining related explicit situations.

The case $X$ is closed (or Theorem \ref{thm:1}) can be regarded to be a case $R_c-q_c(S(c))={\sqcup}_{{\lambda}_{\rm P} \in {\Lambda}_{\rm P}} U_{{\lambda}_{\rm P}}$. Note that this is not a non-proper case. However, the fact that the function is proper is not essential.

\cite{kitazawa} is for Main Theorem, where we may regard
$V_{R_c,{\Lambda}_{\rm P}, {\Lambda}_{\rm NP}}  \subset q_c(S(c))$ and $V_{R_c,{\Lambda}_{\rm P}, {\Lambda}_{\rm NP}}=q_c(S(c))$ in a suitable way, and may not regard $V_{R_c,{\Lambda}_{\rm P}, {\Lambda}_{\rm NP}}-(V_{R_c,{\Lambda}_{\rm P}, {\Lambda}_{\rm NP}}\bigcap q_c(S(c)))$ to be non-empty.
\section{Conflict of interest, data, and so on.}
 The author is a researcher at Osaka Central Advanced Mathematical Institute (OCAMI researcher). It is funded by MEXT Promotion of Distinctive Joint Research Center Program JPMXP0723833165. He is not employed by the institute or the projects there. He thanks all there for their hospitality. 

 
No data other than the present file is generated, related to the present paper.



\begin{thebibliography}{25}
\bibitem{banyagahurtubise} A. Banyaga and D. Hurtubise, \textsl{Lectures on Morse Homology}, Kluwer Texts in the Mathematical Sciences, vol. 29. Kluwer Academic Publishers Groups, Dordrecht (2004).


	%




\bibitem{gelbukh1} I. Gelbukh, \textsl{On the topology of the Reeb graph}, Publicationes Mathematicae Debrecen 104(3--4) (2023), 343--365.
\bibitem{gelbukh2} I. Gelbukh, \textsl{Topological dimension of a Reeb graph and a Reeb space}, Topology and its Applications, Volume 373, 1 November 2025, 109462.
\bibitem{gelbukh3} I. Gelbukh, \textsl{A corrected sufficient condition for the hausdorffness of a Reeb space}, Beitr\"age zur Algebra und Geometrie/ Contribution to Algebra and Geometry 66: 443-335, Volume 373, 1 November 2025, 109462.
\bibitem{gelbukh4} I. Gelbukh, \textsl{On Hausdorffness of a Reeb space and a Reeb graph}, https://www.researchgate.net/profile/Irina-Gelbukh-2/publication/394428146\_On\_Hausdorffness
\_of\_a\_Reeb\_space\_and\_a\_Reeb\_graph/links/\\kk68a28a856327cf7b63d72900/On-Hausdorffness-of-a-Reeb-space-and-a-Reeb-graph.pdf,\\ accepted for publication in Filomat, 2026.




\bibitem{golubitskyguillemin} M. Golubitsky and V. Guillemin, \textsl{Stable Mappings and Their Singularities}, Graduate Texts in Mathematics (14), Springer-Verlag(1974).
\bibitem{hatcher} A. Hatcher, \textsl{Algebraic Topology}, Electronic version, https://pi.math.cornell.edu/$\sim$hatcher/AT/AT+.pdf.
\bibitem{izar} Izar. S. A, \textsl{Funções de Morse e Topologia das Superfícies I: O grafo de Reeb de $f:M \rightarrow \mathbb{R}$}, Métrica no. 31, In Estudo e Pesquisas em Matemática, Brazil: IBILCE, 1988, https://www.ibilce.unesp.br/Home/Departamentos/Matematica/metrica-31.pdf.
\bibitem{kitazawa} N. Kitazawa, \textsl{On Reeb graphs induced from smooth functions on closed or open surfaces}, Methods of Functional Analysis and Topology Vol. 28 No. 2 (2022), 127--143, arXiv:1908.04340.







\bibitem{martinezalfaromezasarmientooliveira} J. Martinez-Alfaro, I. S. Meza-Sarmiento and R. Oliveira, \textsl{Topological  classification of simple Morse Bott functions on surfaces}, Contemp. Math. 675 (2016), 165--179.%
\bibitem{masumotosaeki} Y. Masumoto and O. Saeki, \textsl{A smooth function on a manifold with given Reeb graph}, Kyushu J. Math. 65 (2011), 75--84.

\bibitem{michalak} L. P. Michalak, \textsl{Realization of a graph as the Reeb graph of a Morse function on a manifold}. Topol. Methods in Nonlinear Anal. 52 (2) (2018), 749--762, arXiv:1805.06727.
%
\bibitem{milnor1} J. Milnor, \textsl{Morse Theory}, Annals of Mathematic Studies AM-51, Princeton University Press; 1st Edition (1963.5.1).
\bibitem{milnor2} J. Milnor, \textsl{Lectures on the h-cobordism theorem}, Math. Notes, Princeton Univ. Press, Princeton, N.J. 1965.
\bibitem{reeb} G. Reeb, \textsl{Sur les points singuliers d\'{}une forme de Pfaff compl\'{e}tement int\`{e}grable ou d\'{}une fonction num\'{e}rique}, Comptes Rendus
 Hebdomadaires des S\'{e}ances de I\'{}Acad\'{e}mie des Sciences 222 (1946), 847--849.
\bibitem{saeki1} O. Saeki, \textsl{Reeb spaces of smooth functions on manifolds}, International Mathematics Research Notices, maa301, Volume 2022, Issue 11, June 2022, 3740--3768, https://doi.org/10.1093/imrn/maa301, arXiv:2006.01689.
\bibitem{saeki2} O. Saeki, \textsl{Reeb spaces of smooth functions on manifolds II}, Res. Math. Sci. 11, article number 24 (2024), https://link.springer.com/article/10.1007/s40687-024-00436-z.
\bibitem{sharko} V. Sharko, \textsl{About Kronrod-Reeb graph of a function on a manifold}, Methods of Functional Analysis and
 Topology 12 (2006), 389--396.
\bibitem{willard} S. Willard, \textsl{General Topology}, Courier Corporation, 2012.
	
\end{thebibliography}
\end{document}